\documentclass[a4paper,10pt]{article}

\usepackage{amsmath,amsthm}
\usepackage{amssymb}
\usepackage[all]{xy}

\newtheorem{theorem}{Theorem}[section]

\newtheorem{proposition}[theorem]{Proposition}
\newtheorem{corollary}[theorem]{Corollary}
\newtheorem{conjecture}[theorem]{Conjecture}
\newtheorem{remark}[theorem]{Remark}

\title{On Automorphisms and Subtowers of an asymptotically optimal Tower of Function Fields}
\author{Thorsten Lagemann\footnote{Universit\"at Paderborn, lagemann@math.uni-paderborn.de}}

\date{May 2007}

\begin{document}

\maketitle

\parindent0mm

\begin{abstract}
In this article we investigate the automorphism group of an asymptotically optimal tower of function fields introduced by Garcia and Stichtenoth. In particular we provide a detailed description of the decomposition group of some rational places. This group acts 
on the algebraic-geometric standard codes obtained by the Garcia-Stichtenoth tower exceeding the Gilbert-Varshamov 
bound. The fields fixed by the decomposition groups form an asymptotically optimal non-Galois subtower, which has been first found by Bezerra and Garcia and yields an improvement for computing codes above the Gilbert-Varshamov bound. In this article we also describe its proportionality to the Garcia-Stichtenoth tower and obtain new precise results on its rational places and their Weierstra\ss\ semigroups.

\end{abstract}

\section*{Introduction}\label{Introduction}

The celebrated theorem of Tsfasman, Vladut and Zink (1982) states the existence of modular curves with optimal asymptotic quotient of the number of rational places to genus, but the proof was not constructive. Only in the nineties Garcia and Stichtenoth discovered explicit descriptions of towers of function fields with this asymptotical optimal behaviour \cite{GarStiASEaDVb,GarStiABST}. In coding theory these towers are of great interest because one can obtain (asymptotically) long codes strictly above the Gilbert-Varshamov bound. In this article we deal with the norm-trace tower  $T_m=K(x_0,\ldots,x_m)$ introduced in \cite{GarStiABST} with constant field $K=\mathbb{F}_{q^2}$ defined by the relations
$$ x_i^q+x_i = \frac{x_{i-1}^q}{x_{i-1}^{q-1}+1} \qquad \textup{for } i=1,\ldots,m. $$
The rational pole $\mathfrak{P}_{\infty}$ of $x_0,\ldots,x_m$ is in the focus of coding theoretic applications, because one can obtain the above mentioned codes by using the Riemann-Roch spaces $\mathcal{L}(\mathfrak{P}_{\infty}^t)$ with $t\in \mathbb{N}$. A subgroup of the automorphism group of these codes is given by the decomposition group  
$$G_m(\mathfrak{P}_{\infty}) = \{ \sigma \in \textup{Aut}(T_m/K) : \sigma(\mathfrak{P}_{\infty})=\mathfrak{P}_{\infty} \} $$ 
of $\mathfrak{P}_{\infty}$. In this article we will compute $G_m(\mathfrak{P}_{\infty})$ and a subgroup of $\textup{Aut}(T_m/K)$ properly containing $G_m(\mathfrak{P}_{\infty})$ which we conjecture to be the entire automorphism group of the norm-trace tower. 
For $m=1$ this coincides with the result of Aleschnikov \cite{AleGB}. In this article we verify our conjecture for $m=2$ in odd characteristic and $m=2,3,4$ in even characteristic. Furthermore we describe the subtower formed by the fixed fields of $G_m(\mathfrak{P}_{\infty})$ which has some interesting applications in coding theory.
\\[0.1cm]
This article is organized as follows. In sections \ref{DecompositonGroup}, \ref{Decomposition} and \ref{FullAutomorphism} we compute the automorphism group $G_m:=\textup{Aut}(T_m/K)$ via its action on the rational places in $T_m/K$. First we determine the stabilizer of $\mathfrak{P}_{\infty}$ which is the decomposition group $G_m(\mathfrak{P}_{\infty})$. 
It turns out that its order and its structure do not depend on $m$ (at least for $m\geq 2$). 
\begin{theorem}\label{Hauptsatz1}
The decomposition group $G_m(\mathfrak{P}_{\infty})$ has order
$$ | G_m(\mathfrak{P}_{\infty}) | = 
\left\{ \begin{array}{cl} 
q(q-1) & \ \text{for}  \  q \ \text{odd} \ \ \text{ or } \  m=1 \\
q^2(q-1) & \ \text{for} \ q \ \text{even}  \ \text{and} \ m\geq 2.
\end{array}\right.
 $$
\end{theorem}

Then we determine several conjugated places $\mathfrak{Q}$ of $\mathfrak{P}_{\infty}$ and the corresponding automorphisms $\tau$ with $\tau(\mathfrak{Q})=\mathfrak{P}_{\infty}$. These places can be described easily and in combination with Theorem \ref{Hauptsatz1} we can establish a lower bound for the cardinality of $\textup{Aut}(T_m/K)$.
\begin{theorem}\label{Hauptsatz2}
All rational places supporting $x_0^q+x_0$ are conjugated to $\mathfrak{P}_{\infty}$. 

In particular, the order of $\textup{Aut}(T_m/K)$ is bounded by
$$ |\textup{Aut}(T_m/K)| \geq 
\left\{ \begin{array}{cl} 
2q^2(q-1) & \ \text{for}  \ q \ \text{odd}  \ \ \text{ or } \  m=1 \\
q^3(q^2-1) & \ \text{for} \ q \ \text{even}  \ \text{and} \ m = 2 \\
2q^4(q-1) & \ \text{for} \ q \ \text{even}  \ \text{and} \ m\geq 3.
\end{array}\right.
$$
\end{theorem}

Thus we have specified a subgroup of $G_m$ properly containing $G_m(\mathfrak{P}_{\infty})$ which we conjecture to be the entire automorphism group of the norm-trace tower $T_m/K$. For $m=1$ and $q\neq 2$ this is the result of Aleschnikov \cite{AleGB}. In the final section \ref{FullAutomorphism} we present our proof for the sharpness of the bounds of Theorem \ref{Hauptsatz2} for $m=2$ and $m=2,3,4$ for even $q$ respectively. In summary we obtain
\begin{theorem}\label{Hauptsatz3}
The automorphism group of the norm-trace tower $T_m/K$ of height $1\leq m \leq 4$ has order
$$ |\textup{Aut}(T_m/K)| =
\left\{ \begin{array}{cl} 
2q^2(q-1) & \ \text{for}  \ q\geq 4 \ \text{ odd}  \ \text{and} \  m=1,2 \\
2q^2(q-1) & \ \text{for}  \ q\geq 4 \ \text{even}  \ \text{and} \  m=1 \\
q^3(q^2-1) & \ \text{for} \ q\geq 4 \ \text{even}  \ \text{and} \ m = 2 \\
2q^4(q-1) & \ \text{for} \ q\geq 4 \ \text{even}  \ \text{and} \ m=3,4.
\end{array}\right.
$$ 
\end{theorem}
In order to establish the proof of Theorem \ref{Hauptsatz3}, we investigate the fixed fields of the exhibited automorphisms, which are interesting in their own right. The focus is on the subtower $Z_m/K$ of $T_m/K$, which is formed by the fixed fields of $G_m(\mathfrak{P}_{\infty})$. It is also an asymptotically optimal tower since all subtowers of $T_m/K$ are optimal by \cite{GarStiABST}. First it has been found and described by Bezerra and Garcia in 2004 \cite{BezGarNGT} and then independently by the author in 2006 \cite{Lagemann}, who has not been aware of \cite{BezGarNGT}. Both descriptions of $Z_m/K$ differs in their approach. Bezerra and Garcia gave a recursion formula for an asymptotically optimal non-Galois tower and stated its relation to the norm-trace tower at the end with \cite[Remark 1]{BezGarNGT}. Vice versa the author used the Galois correspondence to obtain the description of $Z_m/K$. Anyway, the author choose to provide a complete presentation of the latter approach in section \ref{DecompositionTower} for those readers, who might have interests in it. Theorem \ref{5theorem2} and most results of Proposition \ref{5prop2} are already proven in \cite{BezGarNGT}. But we will obtain more precise results on the number of rational places with Theorem \ref{5theorem1} and their Weierstra\ss\ semigroups with Theorem \ref{5theorem3}, which both are not given in \cite{BezGarNGT}. In particular we observe the following proportionality of the here-called decomposition tower to the norm-trace tower.

\begin{theorem}\label{Hauptsatz4}
The fixed fields of $G_m(\mathfrak{P}_{\infty})$ form an asymptotically optimal tower $Z_m/K$ with
$Z_m=K(x_0^{q-1},\ldots,x_m^{q-1})$ and
$$ T_m^{G_m(\mathfrak{P}_{\infty})} = Z_{m-\varepsilon} \qquad \text{for}\quad 
\varepsilon = \left\{ \begin{array}{cl} 
1 & \ \text{for}  \  q \ \text{odd} \ \ \text{ or } \  m=1 \\
2 & \ \text{for} \ q \ \text{even}  \ \text{and} \ m\geq 2.
\end{array}\right. $$
We call $Z_m/K$ decomposition tower of height $m$. It has properties that are proportional with factor $q-1$ to the corresponding properties of the norm-trace tower $T_m/K$ as follows.
\begin{enumerate}
 \item[\textup{(a)}] The genus of $Z_m/K$ is $g(Z_m/K)=g(T_m/K)/(q-1)$.
 \item[\textup{(b)}] The number of rational places in $Z_m/K$ is 
$$ N_1(Z_m/K) = q^{m+1}+q^2-q+2+\varepsilon^*_m $$
 with $\varepsilon^*_m=0$ for odd $q$ or $m=1$, $\varepsilon^*_2=q$ and $\varepsilon^*_m=2q$ for even $q$ and $m\geq 3$. For $m\gg 0$ this number satifies $N_1(Z_m/K) \sim N_1(T_m/K)/(q-1)$.
 \item[\textup{(c)}] An integer $n$ is a pole number of $\mathfrak{P}_{\infty}$ in $Z_m/K$ if and only if $n\cdot (q-1)$ is a pole number of $\mathfrak{P}_{\infty}$ in $T_m/K$.
\end{enumerate}

\end{theorem}

The ratio $N_1(Z_m/K)/g(Z_m/K)$ of the number of rational places and the genus is slightly better than $N_1(T_m/K)/g(T_m/K)$ and significantly better for $m=1,2$. Therefore the standard codes in $Z_m/K$ lie above the Gilbert-Varshamov bound and are also better than the codes in $T_m/K$. Furthermore the codes in $Z_m/K$ can be computed faster than those in $T_m/K$ since the genus of $Z_m/K$ is smaller than the genus of $T_m/K$.

\section{Notation and Preliminaries}\label{Notations}
We assume that the reader is familiar with \cite{GarStiABST,PellStiTorWS,AleKumShumStiSP}, because many assertions are deduced by properties of the norm-trace tower exhibited in these articles. Throughout this article $K$ denotes a finite field $\mathbb{F}_{q^2}$ of quadratic order. The fibres $\{ a \in K : a^q+a = c \}$ of $c$ by the trace from $K$ to $\mathbb{F}_q$ is denoted by $A_c$. For $c=0$ we write $A=A_0$ and $A^{\times}=A\backslash \{ 0 \}$. The ramification index, relative degree and decomposition index (the number of extensions) of a place $\mathfrak{P}|\mathfrak{P}\cap F$ in an extension $E/F$ are denoted by $e_{\mathfrak{P}}(E/F)$, $f_{\mathfrak{P}}(E/F)$ and $r_{\mathfrak{P}}(E/F)$ respectively. In order to avoid some case distinctions we define 
$\varepsilon=1$ for odd $q$ or $m=1$ and $\varepsilon=2$ for even $q$ and $m\geq 2$.

\begin{proposition}[Ramification and genus of the norm-trace tower]
\ \label{2prop1}

The extensions $T_m/K(x_i)$ are unramified outside $x_0^q+x_0$ for $0\leq i \leq m$. All ramified places of $T_m/T_0$ are listed in the statements (a) and (b). 
\begin{enumerate}
\item[\textup{(a)}] All places supporting $x_0^{q-1}+1$ are totally ramified in $T_m/T_0$. These places are the only ones being totally ramified in $T_m/T_0$.
\item[\textup{(b)}] The zeros of $x_i-a$ with $a\in A^{\times}$ are completely decomposed in $T_{i}/T_0$, unramified in $T_{2i}/T_i$ and totally ramified in $T_m/T_{2i}$. For odd $q$ these zeros have a non-trivial relative degree in $T_{i+1}/T_i$. For even $q$ these zeros are completely decomposed in $T_{i+1}/T_i$ and - in case of $i\geq 2$ - have a non-trivial relative degree in $T_{i+2}/T_{i+1}$. \item[\textup{(c)}] The norm-trace tower $T_m/K$ of height $m$ has genus
$$ g_m = \left\{ \begin{array}{c@{\quad}l} (q^{\frac{m+1}{2}}-1)^2 & \textup{for } m\equiv 1 \pmod{2} \\
 (q^{\frac{m}{2}}-1)(q^{\frac{m+2}{2}}-1) & \textup{for } m\equiv 0 \pmod{2}. \end{array}\right. $$
\item[\textup{(d)}] All zeros of $x_m^q+x_m$ are totally ramified in $T_m/K(x_m)$.
\end{enumerate}
\end{proposition}
For the proof of Proposition \ref{2prop1} see \cite{GarStiABST}. Statement (d) is of interest in conjunction with Proposition \ref{6prop2}. We denote the pole of $x_0^{q-1}+1=\prod\nolimits_{a\in A^{\times}} (x_0-a)$ by $\mathfrak{P}_{\infty}$ and its zeros by $\mathfrak{P}_a$. For each $m\geq 0$ there is exactly one pole of $x_0$ and exactly one zero of $x_0-a$ with $a\in A^{\times}$ respectively. If we intend to stress the membership of these places in $T_m$ we denote these places with an additional index $m$, i.e. $\mathfrak{P}_{\infty,m}$ or $\mathfrak{P}_{a,m}$. The zero of $x_m-b$ for $b\in A$ is totally ramified in $T_m/K(x_m)$ and is denoted by $\mathfrak{Q}_{b,m}$. In odd characteristic these places are the only rational places supporting $x_0^q+x_0$.

\begin{proposition}[Rational places of the norm-trace tower]\ \label{2prop1a}
\begin{enumerate}
 \item[\textup{(a)}] All zeros of $x_0^{q^2}-x_0$ outside $x_0^q+x_0$ are completely decomposed in $T_m/K(x_i)$ for $0\leq i \leq m$. 
 \item[\textup{(b)}] The norm-trace tower of height $m\geq 1$ has 
$$ N_1(T_m/K)=q^{m+2}-q^{m+1} + 2q + \varepsilon_m $$ 
rational places with $\varepsilon_m=0$ for odd $q$ or $m=1$, $\varepsilon_2=q(q-1)$ and $\varepsilon_m=2q(q-1)$ for even $q$ and $m\geq 3$.
\end{enumerate}
\end{proposition}
Statement (a) is proved in \cite{GarStiABST}. A proof of statement (b) for odd characteric and further references are given in \cite{AleKumShumStiSP}.
 
\begin{proposition}[Weierstra\ss\ semigroups of $\mathfrak{P}_{\infty}$]
\ \label{2prop2}

\begin{enumerate}
\item[\textup{(a)}] The place $\mathfrak{P}_{\infty,m}$ has the (inductively defined) Weierstra\ss\ semigroup $\mathbb{H}_m = q\cdot \mathbb{H}_{m-1} \cup \{n \geq c_m\}$ with $\mathbb{H}_0:=\mathbb{N}$ and conductor
$$ c_m = \left\{ \begin{array}{c@{\quad}l} q^{m+1} - q^{\frac{m+1}{2}} & \textup{for } m\equiv 1 \pmod{2} \\
 q^{m+1} - q^{\frac{m+2}{2}} & \textup{for } m\equiv 0 \pmod{2}. \end{array}\right. $$
\item[\textup{(b)}] The Riemann-Roch spaces  $\mathcal{L}(\mathfrak{P}_{\infty,m}^t)$  with $0\leq t \leq q^m(q-1)$ are generated by polynomials in $x_0$ of degree $\leq tq^{-m}$.
\end{enumerate}
\end{proposition}

For the proof of Proposition \ref{2prop2} see \cite{PellStiTorWS}.

\section{The Decomposition Group of $\mathfrak{P}_{\infty}$}\label{DecompositonGroup}

For the rest of the article we present the new results on the automorphism group of $T_m/K$. In this section we determine the decomposition group of $\mathfrak{P}_{\infty}$.

\begin{theorem}[Decomposition group of $\mathfrak{P}_{\infty}$]\ \label{3theorem1}

The decomposition group $G_m(\mathfrak{P}_{\infty})$ of $\mathfrak{P}_{\infty,m}$ has order
$$ | G_m(\mathfrak{P}_{\infty}) | = q^{\varepsilon}(q-1) =
\left\{ \begin{array}{cl} 
q(q-1) & \ \text{for}  \  q \ \text{odd} \ \ \text{ or } \  m=1 \\
q^2(q-1) & \ \text{for} \ q \ \text{even}  \ \text{and} \ m\geq 2.
\end{array}\right.
 $$

It is isomorphic to a semi-direct product $A^{\varepsilon} \rtimes \mathbb{F}_q^{\times}$, where its structure is determined by the form of its elements given below. Any $\sigma \in G_m(\mathfrak{P}_{\infty})$ satisfies
$$ \sigma(x_m)=cx_m+a \quad \text{and} \quad \sigma(x_i) = cx_i \quad \text{for } i=0,\ldots,m-1	 $$
with $a\in A$ and $c\in \mathbb{F}_q^{\times}$ 
for odd $q$ or $m=1$. For even $q$ and $m\geq 2$ any $\sigma\in G_m(\mathfrak{P}_{\infty})$ is given by
$$ \sigma(x_{m-1})=cx_{m-1}+a \quad \text{and} \quad \sigma(x_i) = cx_i \quad \text{for } i=0,\ldots,m-2	$$
and
$$ \sigma(x_m) = cx_m + \frac{a^2}{cx_{m-2}}+b $$
with $a\in A=\mathbb{F}_q$, $b^q+b=a$ and $c\in \mathbb{F}_q^{\times}$.

\end{theorem}
\begin{proof} One easily verifies that the stated maps are automorphisms of $T_m/K$ and that they form a semi-direct product as specified. So it remains to verify that every automorphism of $T_m/K$ fixing $\mathfrak{P}_{\infty}$ is of the form above.

Let $\sigma$ be an automorphism with $\sigma(\mathfrak{P}_{\infty})=\mathfrak{P}_{\infty}$. Then the Riemann-Roch space $\mathcal{L}(\mathfrak{P}_{\infty}^{q^m})$ is also invariant under the action of $\sigma$. By Proposition \ref{2prop2} it is spanned by $1$ and $x_0$. So we get
$$ \sigma(x_0) = cx_0 + d \qquad \textup{with } c \in K^{\times}, d \in K. $$
Furthermore the divisor of $x_0^{q-1}+1$ is invariant under the action of $\sigma$ as well, because its support contains exactly the
totally ramified places of $T_m/T_0$ by Proposition \ref{2prop1}. Indeed, every place $\sigma(\mathfrak{P}_a)$ satisfies 
$$ q^m = e_{\mathfrak{P}_a}(T_m/T_0)= v_{\mathfrak{P}_a}(x_0-a) = v_{\sigma(\mathfrak{P}_a)}(cx_0+d-a) = e_{\sigma(\mathfrak{P}_a)}(T_m/T_0) $$
and therefore $\sigma(\mathfrak{P}_a)$ is totally ramified in $T_m/T_0$ and hence also a zero of $x_0^{q-1}+1$. So $\sigma(x_0^{q-1}+1)$ is a nontrivial function of $\mathcal{L}(\mathfrak{P}_{\infty}^{q^m(q-1)}\prod \mathfrak{P}_a^{-q^m}) = \langle x_0^{q-1}+1 \rangle$ and we get $$ \sigma(x_0^{q-1}+1) = e (x_0^{q-1}+1) \qquad \textup{with } e \in K^{\times}. $$
Comparing the coefficients of
\begin{align*}
ecx_0^q+edx_0^{q-1}+ecx_0+ed & = e(x_0^{q-1}+1)(cx_0+d) = \sigma(x_0^{q-1}+1)\sigma(x_0) \\ 
 & = \sigma(x_0)^q+\sigma(x_0) = c^qx_0^q + cx_0 + d^q+d 
\end{align*}
we get $ed=0$ and $c^q=ec=c$ which implies $d=0$ and $c^{q-1}=e=1$. So every automorphism $\sigma\in G_m(\mathfrak{P}_{\infty})$ has the properties
$$ \sigma(x_0) = cx_0 \qquad\textup{with } c\in \mathbb{F}_{q}^{\times}.$$
and
\begin{center} $ \sigma(x_1^q+x_1) = \sigma(\frac{x_0^q}{x_0^{q-1}+1}) 
= \frac{\sigma(x_0)^q}{\sigma(x_0)^{q-1}+1} = c \; \frac{x_0^q}{x_0^{q-1}+1} = c \; (x_1^q+x_1). $ \end{center}
Hence we get
$$ \sigma(x_1) - cx_1 = cx_1^q-\sigma(x_1)^q = (cx_1-\sigma(x_1))^q. $$
This function is constant and equals a constant $a$ satisfying $a^q+a=0$. Hence we get
$$ \sigma(x_1)=cx_1+a \qquad \textup{with } a\in A. $$ 
In particular we have proved our hypothesis for $m=1$. 
\\[0.1cm]
\textit{Case I:} Let $q$ be odd and $m\geq 2$. Suppose $a \neq 0$. Then $\sigma(x_1)$ has only non-rational zeros, because $c^{-1}a\in A^{\times}$ holds and all zeros of $x_1-c^{-1}a$ have a non-trivial relative degree in $T_2/T_1$ by \ref{2prop1}(b). But $x_1$ has rational zeros, as for example, the zeros of $x_m^{q-1}+1$ are (rational) zeros of $x_1$. 
Therefore $a$ cannot be an element of $A^{\times}$ and hence $a=0$ holds. This yields
$$ \sigma(x_1) = cx_1 \qquad \textup{if } m \geq 2.$$ 
Inductively we get 
$$ \sigma(x_i) = cx_i \qquad \textup{for } i=1,\ldots,m-1. $$
Actually $\sigma(x_{i-1})=cx_{i-1}$ implies $\sigma(x_i^q+x_i)=c(x_i^q+x_i)$ and therefore $\sigma(x_i)=cx_i+a$ with $a\in A$. The zeros of $cx_i+a$ with $a\in A^{\times}$ are non-rational because of their non-trivial relative degree in $T_{i+1}/T_i$ and $x_i$ has rational zeros. So only $\sigma(x_i)=cx_i$ is possible. For $i=m$ however $a\neq 0$ is possible. Actually the Galois group of $T_m/T_{m-1}$ contains the maps with $x_m \mapsto x_m + a$. Finally we get
$$ \sigma(x_m) = cx_m + a \qquad \textup{with } a \in A.$$
\textit{Case II.1:} Let $q$ be even and $m=2$. The above argument cannot be applied here, because all zeros of $cx_1+a$ are completely decomposed in $T_2/T_1$. But in this case $A=\mathbb{F}_q$ holds and we get
\begin{align*}
\sigma (x_2^q+x_2) & = \frac{ \sigma (x_1)^{q+1} }{\sigma (x_1^q+x_1) } 
= \frac{ (cx_1^q+a)(cx_1+a) }{c(x_1^q+x_1)}  = \frac{ c^2 x_1^{q+1} + ca (x_1^q+x_1) + a^2 }{c(x_1^q+x_1)} \\
& =  c (x_2^q+x_2) + a + \frac{a^2}{c} \frac{1}{x_1^q+x_1} \\
& = c (x_2^q+x_2) + a + \frac{a^2}{c} \left( \left(\frac{1}{x_0}\right)^q + \frac{1}{x_0} \right) \\
& = \left( cx_2 + \frac{a^2}{cx_0} + b \right)^q + \left( cx_2 + \frac{a^2}{cx_0} + b \right)
 \qquad\qquad\textup{with } b \in A_a.
\end{align*}
Hence it follows that
\begin{center}$ \sigma(x_2) - \left( cx_2 + \frac{a^2}{cx_0} \right) \in b+A=A_a. $
\end{center}
This is just the property as stated for even characteristic.
\\[0.1cm]
\textit{Case II.2:} Let $q$ be even and $m\geq 3$. Because of
$$ \sigma(x_1^q+x_1) = c(x_1^q+x_1) $$
$\sigma$ permutes the zeros of $x_1^q+x_1$. But the zeros of $x_1$ are completely decomposed in $T_3/T_2$ and the zeros of $x_1^{q-1}+1$ are totally ramified in $T_m/T_2$. Therefore the cardinality of the support of $x_1$ and $cx_1+a$ with $a\in A^{\times}$ differs. Consequently the zerodivisor of $x_1$ is invariant under $\sigma$ and also the zerodivisor of $x_1^{q-1}+1$. We obtain $\sigma(x_1)=cx_1$. Now we can use the argument of case \textit{I}, because the zeros of $cx_i+a$ have a non-trivial relative degree in $T_{i+2}/T_{i+1}$. By induction we get 
$$ \sigma (x_i)=cx_i \qquad\textup{for } i=1,\ldots,m-2. $$
As in case \textit{II.1} we conclude
$$ \sigma(x_{m-1})=cx_{m-1}+a \qquad\textup{with } a\in A$$  
and $$\sigma(x_m)=cx_m + \frac{a^2}{cx_{m-2}} + b \qquad\textup{for } b\in A_a.$$
This finishes the proof. 
\end{proof}

\section{Decomposition of $\mathfrak{P}_{\infty}$}\label{Decomposition}

The automorphism group $G_m=\textup{Aut}(T_m/K)$ acts on the rational places of the norm-trace tower $T_m/K$. In the preceding section we have established the stabilizer of $\mathfrak{P}_{\infty}$. Now we just need to determine the conjugated places of $\mathfrak{P}_{\infty}$ in order to quantify the cardinality of $G_m$. 
Of course the number of these conjugated places is the decomposition index $r_{\infty}:=r_{\mathfrak{P}_{\infty}}(T_m/T_m^{G_m})$ and $|G_m|=r_{\infty}\cdot |G_m(\mathfrak{P}_{\infty})|$ holds.

\begin{proposition}\label{4prop1}
All rational places supporting $x_0^q+x_0$ are conjugated to $\mathfrak{P}_{\infty}$. 
In particular, the decomposition index $r_{\infty}$ of $\mathfrak{P}_{\infty}$ is bounded by
$$ r_{\infty} \geq 
\left\{ \begin{array}{cl} 
2q & \ \text{for}  \ q \ \text{odd}  \ \ \text{ or } \  m=1 \\
q(q+1) & \ \text{for} \ q \ \text{even}  \ \text{and} \ m = 2 \\
2q^2 & \ \text{for} \ q \ \text{even}  \ \text{and} \ m\geq 3.
\end{array}\right.
$$
\end{proposition}
\begin{proof} We present automorphisms $T_m/K$ which send the above mentioned places to $\mathfrak{P}_{\infty}$. For the zero $\mathfrak{P}_a$ of $x_0-a$ with $a\in A^{\times}$ we find the automorphism $\tau$ with
$$ \tau(x_0) = \frac{ax_0}{x_0+b}, \quad \tau(x_i)=ab^{-1}x_i \quad\textup{for } i=1,\ldots,m-2 $$
and
$$ \tau(x_{m-1})=ab^{-1}x_{m-1}, \quad \tau(x_m)=ab^{-1}x_m + d \quad\textup{for odd } q $$
or
$$ \tau(x_{m-1})=ab^{-1}x_{m-1}+d, \quad \tau(x_m)=ab^{-1} + \frac{d^2}{ab^{-1}x_{m-2}} + e \quad\textup{for even } q$$
respectively with $b\in A^{\times}, d\in A$ and $e\in A_d$, which sends $\mathfrak{P}_a$ to $\mathfrak{P}_{\infty}$. With these properties we can describe all automorphisms with $\tau(\mathfrak{P}_a)=\mathfrak{P}_{\infty}$. 
\\[0.1cm]
A rational zero $\mathfrak{Q}_{b,m}$ of $x_0$ is uniquely determined as zero of $x_m-b$ with $b\in A$. An automorphism with $\tau(\mathfrak{Q}_{b,m})=\mathfrak{P}_{\infty}$ is given by
$$ \tau(x_m)=\frac{c+bx_0}{x_0} \quad\textup{and}\quad \tau(x_i)=\frac{c}{x_{m-i}} \quad\textup{for } i=0,\ldots,m-1 $$
with $c\in\mathbb{F}_q^{\times}$. (This automorphism reflects the pyramide of the norm-trace tower.) Therefore $\mathfrak{Q}_{b,m}$ is conjugated to $\mathfrak{P}_{\infty}$.  For odd characteristic we have considered all rational places supporting $x_0^q+x_0$. 
\\[0.1cm]
For even characteristic $x_0^q+x_0$ has also rational zeros that are zeros of $x_1^{q-1}+1$ or $x_{m-1}^{q-1}+1$ respectively. A zero $\mathfrak{Q}$ of $x_{m-1}+a$ is uniquely determined by
$$ \tilde{x}_m = x_m + \frac{a^2}{x_{m-2}} + b \in \mathfrak{Q} \quad\textup{with } b\in A_a. $$
This holds by the proof of \cite[Lemma 3.4]{GarStiABST}. We can find some automorphism $\sigma$ of $G_m(\mathfrak{P}_{\infty})$ sending $\tilde{x}_m$ to $x_m$ and hence $\mathfrak{Q}_{0,m}$ to $\mathfrak{Q}$. Therefore $\mathfrak{P}_{\infty},\mathfrak{Q}_{0,m}$ and $\mathfrak{Q}$ are conjugated. A zero $\mathfrak{Q}$ of $x_1+a$ is uniquely determined by
$$ \tilde{x}_2 = x_2 + \frac{a^2}{x_0} + b \in\mathfrak{Q} \quad\textup{with } b \in A_a. $$
We find an automorphism $\rho$ with $\rho(\mathfrak{Q}_{0,m})=\mathfrak{Q}$ as following. First we choose a mapping $\sigma$ of $G_m(\mathfrak{P}_{\infty})$ with 
$$ \sigma(x_m)=\tilde{x}_m = a^2x_m + \frac{1}{x_{m-2}} + b.  $$
The above described automorphism $\tau$ with $\tau(\mathfrak{Q}_{0,m})=\mathfrak{P}_{\infty}$ and $c=1$ sends $\tilde{x}_m$ to 
$$ \tau(\tilde{x}_m) = \frac{a^2}{x_0} + x_2 + b = \tilde{x}_2 \in \mathfrak{Q}. $$
The composition $\rho = \tau \circ \sigma$ satisfies $\rho(\mathfrak{Q}_{0,m})=\mathfrak{Q}$ as desired. 
\\[0.1cm]
Finally we get the bounds of $r_{\infty}$ by counting the rational places in the support of $x_0^q+x_0$.
\end{proof}

\section{The Decomposition Tower}\label{DecompositionTower}

In this section we investigate some fixed fields of the automorphisms presented in Theorem \ref{3theorem1} and Proposition \ref{4prop1}. We will see that the decomposition fields of $\mathfrak{P}_{\infty}$ form a subtower $Z_m/K$ of the norm-trace tower $T_m/K$ which is generated by a Kummer descent of degree $q-1$. This subtower inherits good properties of the norm-trace tower and it turns out that its genus, number of rational places and pole numbers of $\mathfrak{P}_{\infty}$ are proportional with factor $q-1$ to those in $T_m/K$.

\begin{theorem}[Decomposition tower $Z_m/K$]\ \label{5theorem0}\label{5remark1}\label{5prop0}\label{5prop1}
 
The fixed fields of $G_m(\mathfrak{P}_{\infty})$ form a subtower of the norm-trace tower which we call the decomposition tower $Z_m/K$. It is generated by the $(q-1)$-th power of $x_0,\ldots,x_m$, i.e.
$$ Z_m=K(z_0,\ldots,z_m) \qquad \text{with } z_i=x_i^{q-1}. $$
Following assertions hold:
\begin{enumerate}
 \item[\textup{(a)}] The decomposition field of $\mathfrak{P}_{\infty,m}$ in $T_m/T_m^{G_m}$ is
	$$ T_m^{G_m(\mathfrak{P}_{\infty})} = Z_{m-\varepsilon} \qquad \text{for } m\geq 1. $$
 \item[\textup{(b)}] The fixed field of the group $G_m(\mathfrak{P}_{\infty}\prod\mathfrak{P}_a)$ generated by the automorphisms of $T_m/K$ permuting the support of $x_0^{q-1}+1$ is 
	$$ Z_{m-\varepsilon-1}^1 = K(z_1,\ldots,z_{m-\varepsilon}) \qquad \textup{for }m\geq 1 + \varepsilon. $$
 \item[\textup{(c)}] The defining relations for $z_0,\ldots,z_m$ are
	$$ z_{i+1}(z_{i+1}+1)^{q-1} = \frac{z_i^q}{(z_i+1)^{q-1}}. $$
 \item[\textup{(d)}] The decompostion tower $Z_m/K$ satisfies the equalities
	\begin{center} $Z_m=Z_{m-1}(\frac{x_m}{x_{m-1}})=Z_0(\frac{x_1}{x_0},\ldots,\frac{x_m}{x_{m-1}})
=K(x_0^{q-1},\frac{x_1}{x_0},\ldots,\frac{x_m}{x_{m-1}}).$\end{center}
	The minimal relation for the extension $Z_{i+1}/Z_i$ is
	$$ \left( \frac{x_{i+1}}{x_i} \right)^q + \frac{1}{z_i} \left(\frac{x_{i+1}}{x_i}\right) = \frac{1}{z_i+1} $$
	and the minimal relation for $Z_{i+1}/K(z_1,\ldots,z_{i+1})$ is
	$$ \left(\frac{x_{i+1}}{x_i} \right)^q  + z_{i+1} \left(\frac{x_{i+1}}{x_i} \right) = \frac{z_{i+1}}{z_{i+1}+1}. $$
	In particular $T_m=Z_m(x_i)$ holds for $0\leq i \leq m$ and $Z_m/K$ is a subtower of the norm-trace tower of index $[T_m:Z_m]=q-1$.
\end{enumerate}

\end{theorem}
\begin{proof} 
(a) The decomposition field $T_m^{G_m(\mathfrak{P}_{\infty})}$ of $\mathfrak{P}_{\infty,m}$ contains $Z_{m-\varepsilon}$ because of
$$ \sigma(z_i) = \sigma(x_i)^{q-1} = c^{q-1}x_i^{q-1} = z_i $$
for $\sigma\in G_m(\mathfrak{P}_{\infty})$ and $i=0,\ldots,m-\varepsilon$. The reverse inclusion follows by (d) due to dimension reasons, i.e. $$ [T_m:Z_{m-\varepsilon}] = q^{\varepsilon}(q-1) = |G_m(\mathfrak{P}_{\infty})|. $$
(b) Any automorphism $\tau$ of $G_m(\mathfrak{P}_{\infty}\prod\mathfrak{P}_a)\backslash G_m(\mathfrak{P}_{\infty})$ has the properties
$$ \tau(x_0) = \frac{ax_0}{x_0+b} \quad\textup{and}\quad \tau(x_i)=ab^{-1}x_i \quad\textup{for } i=1,\ldots,m-\varepsilon $$
with $a,b\in A^{\times}$. Hence we obtain $\tau(z_i)=z_i$ for $1\leq i \leq m-\varepsilon$ by $(ab^{-1})^{q-1}=1$ and therefore $Z_{m-\varepsilon}$ is contained in the fixed field of $G_m(\mathfrak{P}_{\infty}\prod\mathfrak{P}_a)$. The reverse inclusion also follows by (d) due to $$[T_m:Z_{m-\varepsilon-1}]=q^{\varepsilon+1}(q-1)=|G_m(\mathfrak{P}_{\infty}\prod\mathfrak{P}_a)| .$$
(c) We get relative relations for $z_0,\ldots,z_m$ due to the $(q-1)$-th power of the relations of the norm-trace tower
$$ z_{i+1}(z_{i+1}+1)^{q-1} = (x_{i+1}^q+x_{i+1})^{q-1} = \left( \frac{x_i^q}{x_i^{q-1}+1}\right)^{q-1} 
= \frac{z_i^q}{(z_i+1)^{q-1}}. $$ 
In particular we get $[Z_{i+1} : Z_i]\leq q$ and $[T_0:Z_0]=q-1$. 
\\[0.1cm]
(d) The stated equalites are obtained by 
\begin{align*}
 \frac{x_i}{x_{i-1}} \
& = \frac{x_i^q+x_i}{x_i^{q-1}+1}\cdot \frac{1}{x_{i-1}}
= \frac{x_{i-1}^q}{x_{i-1}^{q-1}+1} \cdot \frac{1}{x_{i-1}} \cdot \frac{1}{x_i^{q-1}+1} \\
& = \frac{x_{i-1}^{q-1}}{(x_i^{q-1}+1)(x_{i-1}^{q-1}+1)} = \frac{z_{i-1}}{(z_i+1)(z_{i-1}+1)} 
\end{align*}
and
$$ z_i = x_i^{q-1} = \left( \frac{x_i}{x_{i-1}} \right)^{q-1} \cdot x_{i-1}^{q-1} = \left( \frac{x_i}{x_{i-1}} \right)^{q-1} \cdot z_{i-1}. $$
In particular $[Z_{i+1}:Z_i]=q$ and $[T_i:Z_i]=q-1$ follows. 
\end{proof}

Figure \ref{Fig1} shows parts of the Galois correspondence of the norm-trace tower or the decomposition tower respectively.
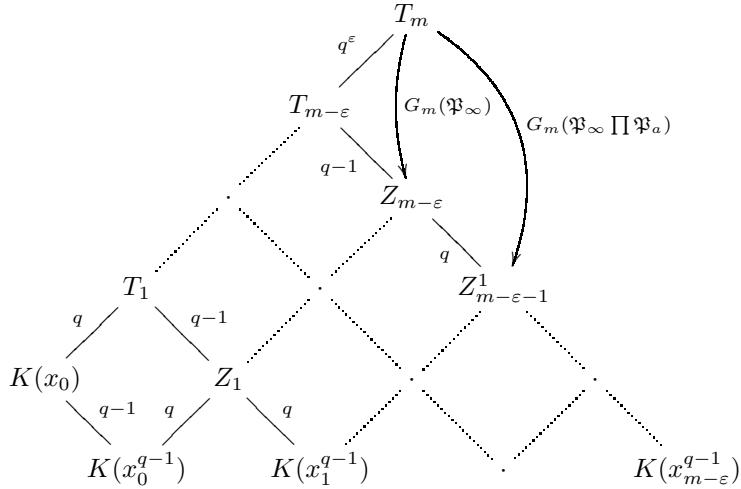
\begin{figure}[hbt]
$$ \xymatrix@!@-3pc
{
 & & & & T_m\ar@/_/[dd]^{G_m(\mathfrak{P}_{\infty})}\ar@/^2pc/[dddr]^{G_m(\mathfrak{P}_{\infty}\prod\mathfrak{P}_a)} & \\
 & & & T_{m-\varepsilon}\ar@{-}[ur]^{q^{\varepsilon}} & \\
 & & \cdot\ar@{.}[ur] & & Z_{m-\varepsilon}\ar@{-}[ul]^{q-1} & & & &  \\
 & T_1\ar@{.}[ur] & & \cdot\ar@{.}[ur]\ar@{.}[ul] & & Z^1_{m-\varepsilon-1}\ar@{-}[ul]^{ q}& \\
 K(x_0)\ar@{-}[ur]^q & & Z_1\ar@{-}[ul]_{ q-1}\ar@{.}[ur]& & 
 \cdot\ar@{.}[ul]\ar@{.}[ur]& & \cdot\ar@{.}[ul] \\
 & K(x_0^{q-1})\ar@{-}[ul]_{q-1}\ar@{-}[ur]^{ q} & & K(x_1^{q-1})\ar@{-}[ul]_{ q}\ar@{.}[ur] 
 & & \cdot\ar@{.}[ul]\ar@{.}[ur] & & K(x_{m-\varepsilon}^{q-1})\ar@{.}[ul]
} 
$$
\caption{Pyramide structure of the decomposition tower}\label{Fig1}
\end{figure}

\begin{proposition}\label{5prop2}
Every rational place of $Z_0/K$ has a rational extension in $Z_m/K$. The following assertions hold:
\begin{enumerate}
\item[\textup{(a)}] The pole and the numerator of $z_0+1$ totally ramify in $Z_m/Z_0$.
\item[\textup{(b)}] The zeros of $z_0-d$ with $d=b^{q-1}\neq 0,1$ completely decompose in $Z_m/Z_0$.
\item[\textup{(c)}] The zerodivisor of $z_0-d$ with $d\neq b^{q-1}$ has exactly one rational extension $\mathfrak{R}$ in $Z_m/K$ and $z_m \not\equiv b^{q-1} \pmod{\mathfrak{R}}$ holds for all $b\in K$.
\item[\textup{(d)}] The functions $z_0,\ldots,z_{m-1}$ have exactly two common rational zeros $\mathfrak{Q}^*_{0,m},$ $\mathfrak{Q}^*_{-1,m}$ in $Z_m/K$ and $z_m \equiv b \pmod{\mathfrak{Q}^*_{b,m}} $ holds. For odd $q$ these are the only rational zeros of $z_0$. For even $q$ and $m\geq 2$ the set of rational zeros of $x_0$ also includes $q$ rational zeros of $z_1+1$ and $q$ rational zeros of $z_{m-1}+1$.
 
\end{enumerate}
 
\end{proposition}
\begin{proof} (See also \cite{BezGarNGT}) (a) This statement is clear, because the support of $z_0+1$ is totally ramified in $T_m/T_0$. The pole $\mathfrak{P}_{\infty,m}$ is even totally ramified in $T_i/Z_i$ for $i=0,\ldots,m$.
\\[0.1cm]
(b) Let $\mathfrak{R}$ be a zero of $z_0-d$ in $Z_m/K$. The primitive element $x_0$ of $T_i/Z_i$ is integral over $\mathfrak{R}\cap Z_i$ for $i=0,\ldots,m$ and its minimal polynomial decomposes modulo $\mathfrak{R}\cap Z_i$ in
$$ T^{q-1}-z_0 \equiv T^{q-1}-d = \prod\limits_{c\in\mathbb{F}_q^{\times}} (T-cb) \pmod{\mathfrak{R}\cap Z_i}. $$
The factors are pairwise different and hence $\mathfrak{R}$ is completely decomposed in $T_i/Z_i$ by the Theorem of Kummer (see \cite[Theorem III.3.7]{Sti}). The zeros of $x_0-cb$ are completely decomposed in $T_m/T_0$ because of $cb\not\in A$. Therefore $\mathfrak{R}$ is completely decomposed in $Z_m/Z_0$.
\\[0.1cm]
(c) Let $\mathfrak{R}$ be a rational zero of $z_0-d$ in $Z_m/K$. We consider the minimal polynomial of $x_1/x_0$ modulo $\mathfrak{R}\cap Z_0$ 
$$ \varphi_0(T)= T^q + d^{-1}T - (d+1)^{-1}. $$ 
We claim that $\tilde{d}=s\cdot (d+1)^{-1} \cdot (d^q+1)^{-1}$ is the only zero of $\varphi_0(T)$ contained in $K$. It is an element of $K$, because $d/\tilde{d}=(d+1)(d^q+1)$ is equal to its $q$-th power. It is a zero of $\varphi_0(T)$ due to
$$ (d+1)\varphi_0(\tilde{d}) = (d+1)(\tilde{d}^q+d^{-1}\tilde{d}) - 1 = (d+1) \;\tilde{d}/d\; (d^q+1) - 1 = 0. $$
Then
$$ 0 = \frac{\varphi_0(\tilde{d}) - \varphi_0(d')}{\tilde{d}-d'} = (\tilde{d}-d')^{q-1} + d^{-1} $$
for some other zero $d'\in \overline{K}$ of $\varphi_0(T)$ not equal to $d$. Hence $d$ is a $(q-1)$-th power of $(\tilde{d}-d')^{-1}$. With our assumption $d \neq b^{q-1}$ for $b\in K$ we conclude that $\tilde{d}$ is the only zero of $\varphi_0(T)$ contained in $K$. By the Theorem of Kummer we get exactly one rational extension and several non-rational extensions of $(z_0-d)_0$ in $Z_1/K$. The rational extension $\mathfrak{R}_1$ satisfies
$$ z_1 = \left(\frac{x_1}{x_0}\right)^{q-1} \negthickspace\cdot z_0 \quad\equiv\quad \tilde{d}^{q-1} \cdot d  \pmod{\mathfrak{R}_1}. $$
The constant $d_1:=\tilde{d}^{q-1}d$ is a $(q-1)$-th power in $K$ if and only if $d$ is a $(q-1)$-th power in $K$. Hence we get $d_1\neq b^{q-1}$ for all
$b\in K$ and the minimal polynomial of $x_2/x_1$ reduced by $\mathfrak{R}_1$ 
$$\varphi_1(T)=T^q+d_1^{-1}T-(d_1+1)^{-1}\equiv T^q+z_1^{-1}T-(z_1+1)^{-1} \pmod{\mathfrak{R}_1}$$ 
also has exactly one zero contained in $K$. Therefore $\mathfrak{R}_1$ has exactly one rational extension $\mathfrak{R}_2$ in $Z_2/K$. Proceeding inductively as above we conclude the proof of (c).

(d) Every zero of $z_0=x_0^{q-1}$ is a zero of $x_0$ and so it is either a common zero of $x_0,\ldots,x_m$ or a common zero of $x_0,\ldots,x_{i-1}$ and $x_i-a$ for some $1\leq i \leq m$ and $a\in A^{\times}$. So every zero of $z_0$ is either a common zero of $z_0,\ldots,z_m$ or a common zero of $z_0,\ldots,z_{i-1}$ and $z_i+1=z_i-a^{q-1}$. We verify inductively that the common zero of $z_0,\ldots,z_{i-1}$ has exactly two extension $\mathfrak{Q}^*_{0,i}$ and $\mathfrak{Q}^*_{-1,i}$ in $Z_i/Z_{i-1}$ with 
$z_i\equiv b \pmod{\mathfrak{Q}^*_{b,m}}$ and
$$ e_{\mathfrak{Q}^*_{0,i}}(Z_i/Z_{i-1}) = 1 \quad\textup{and}\quad e_{\mathfrak{Q}^*_{-1,i}}(Z_i/Z_{i-1}) = q-1 . $$
This assertion is trivially true for $i=0$. For the induction step $i-1$ to $i$ we assume that the assertion is true for $k\leq i-1$ with $i\geq 1$. Obviously $\mathfrak{Q}^*_{0,i-1}\cap Z_0$ totally ramifies in $T_0/Z_0$ as zero of $z_0$. With our induction assumption we get the ramification diagram as in Figure \ref{Fig2}.

\begin{figure}[hb]
$$ \xymatrix@!@-1.8pc{
& & & T_i \\
& & T_{i-1}\ar@{-}[ur]^1 & & Z_i\ar@{-}[ul]_?\\
& T_1\ar@{.}[ur]^1 & & Z_{i-1}\ar@{-}[ul]_{q-1}\ar@{-}[ur]_?\\
T_0\ar@{-}[ur]^1 & & Z_1\ar@{-}[ul]_{q-1}\ar@{.}[ur]_1 \\
 & Z_0 \ar@{-}[ul]_{q-1}\ar@{-}[ur]_1 
}
$$
\caption{Ramification diagram of $\mathfrak{Q}|\mathfrak{Q}^*_{0,i-1}$}\label{Fig2}
\end{figure}
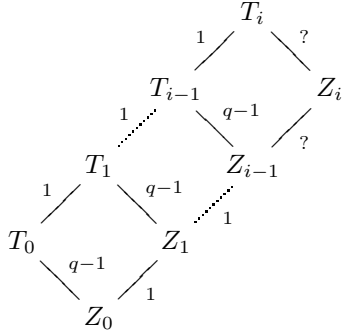

There are $q$ extensions $\mathfrak{Q}_{b,i}$ of $\mathfrak{Q}^*_{0,i-1}$ in $T_i/K$ with $b\in A$ and $z_i = x_i^{q-1} \equiv b^{q-1} \pmod{\mathfrak{Q}_{b,i}}$. The places $\mathfrak{Q}_{a,i}$ with $a\in A^{\times}$ are extensions of a zero of $z_i^{q-1}+1$ by the Theorem of Kummer. Therefore we get $e_{\mathfrak{Q}_{a,i}}(T_i/F_i)=1$ and
$$ e_{\mathfrak{Q}_{a,i}} (F_i/F_{i-1}) = e_{\mathfrak{Q}_{a,i}}(T_i/F_i)\cdot e_{\mathfrak{Q}_{a,i}}(F_i/F_{i-1})
= e_{\mathfrak{Q}_{a,i}}(T_i/F_{i-1}) = q-1. $$
By the arithmetic formula $n = \sum e_jf_j$ (see \cite[Theorem III.1.11]{Sti}) we conclude $e_{\mathfrak{Q}_{0,i}}(F_i/F_{i-1})=1$
and $e_{\mathfrak{Q}_{0,i}}(T_i/F_i)=q-1$. In particular $\mathfrak{Q}^*_{0,i-1}$ has exactly two extensions $\mathfrak{Q}^*_{0,i}$ and $\mathfrak{Q}^*_{-1,i}$ in $Z_i/K$ with $z_i \equiv b \pmod{\mathfrak{Q}_{b,i}^*}$.
This proves our assertion. 

The places $\mathfrak{Q}_{0,m}|\mathfrak{Q}^*_{0,m}$ and $\mathfrak{Q}_{a,m}|\mathfrak{Q}^*_{-1,m}$ are rational in $T_m/K$ and hence $\mathfrak{Q}^*_{0,m}$ and $\mathfrak{Q}^*_{-1,m}$ are rational in $Z_m/K$. The places $\mathfrak{Q}^*_{-1,i}$ with $1\leq i \leq m-1$ completely decompose in $T_{i+k}/Z_{i+k}$ by the Theorem of Kummer due to $T_{i+k}=Z_{i+k}(x_i)$ with $k\geq 0$. Hence the inertia indices of $\mathfrak{Q}^*_{-1,i}$ equal the inertia indices of their extensions in the norm-trace tower. This proves our statement.
\end{proof}

Now we can establish the number of rational places and the genus of $Z_m/K$.

\begin{theorem}[Rational places of the decomposition tower]
\ \label{5theorem1}

The decomposition tower $Z_m/K$ of height $m\geq 1$ has
$$ N_1(Z_m/K) = q^{m+1}+q^2-q+2+\varepsilon^*_m $$
rational places with $\varepsilon^*_m=\varepsilon_m/(q-1)$. (See Proposition \ref{2prop1a} for the definition of $\varepsilon_m$.)
\end{theorem}
\begin{proof} Proposition \ref{5prop2} yields a complete view over the rational places of $Z_m/K$. In the parts (a) and (d) we find $4+\varepsilon^*_m$ rational places. It remains to consider the rational zeros of $z_0-d$ with $d\neq 0,-1$. The zero divisor of $z_0-d$ with $d=b^{q-1}$ for some $b\in K$ is completely decomposed in $Z_m/Z_0$ and has $q^m$ rational prime divisors. For $b\neq b^{q-1}$ there is just one rational place dividing $(z_0-d)_0$. With $s:=\#\{ d\in K : d=b^{q-1}\neq 0,-1 \}$ we get
$$ N_1(Z_m/K) = s\cdot q^m + (q^2-2-s) + 4 + \varepsilon^*_m. $$
Due to \ $(b^{q-1})^{q+1}=b^{q^2-1}=1$ \ a $(q-1)$-th power is a zero of
$$ \left. T^q - T^{q-1} + T^{q-2} \mp \cdots + T -1 = \frac{T^{q+1}-1}{T+1} \ \right| \ (T^{q+1})^{q-1} - 1 = 
\prod\limits_{b\in K^{\times}} (T-b). $$
Hence $s=q$ fulfills. This finishes the proof.
\end{proof}
\begin{theorem}[Genus of the decomposition tower]
\ \label{5theorem2}

The decomposition tower $Z_m/K$ of height $m$ has genus $g_m/(q-1)$, where $g_m$ denotes the genus of the norm-trace tower.
\end{theorem}
\begin{proof} (See also \cite[Lemma 4]{BezGarNGT}) By the Huwitz formula we obtain the genus $g_m^*$ of $Z_m/K$ due to
$$ [T_m:Z_m](g_m^*-1) = g_m - \frac{\mathfrak{D}(T_m/Z_m)}{2} - 1. $$
According to the theory of Kummer extensions (see \cite[Proposition III.7.3]{Sti}) only poles and zeros of $z_0$ ramify in $T_m/Z_m$ with index
$$ e_{\mathfrak{P}}(T_m/Z_m) = \frac{[T_m:Z_m]}{\textup{gcd}([T_m:Z_m],v_{\mathfrak{P}\cap Z_m}(z_0) )} $$
where $\textup{gcd}(\cdot,\cdot)$ denotes the positive greatest common divisor. The resulting different exponent is
$$ d_{\mathfrak{P}}(T_m/Z_m) = e_{\mathfrak{P}}(T_m/Z_m)-1. $$
The pole $\mathfrak{P}_{\infty}$ of $z_0$ is totally ramified in $T_m/Z_0$ and has different exponent $d_{\mathfrak{P}_{\infty}}=q-2$. The only zero of $z_0$ with $\textup{gcd}([T_m:Z_m],v_{\mathfrak{Q}\cap Z_m}(z_0) )\neq q-1$ is $\mathfrak{Q}|\mathfrak{Q}_{0,m}^*$ because the other zeros are ramified in $Z_1/Z_0$ with index $q-1$ by Proposition \ref{5prop2}. Hence $d_{\mathfrak{Q}}(T_m/Z_m)=q-2$ and
$$ \deg(\mathfrak{D}(T_m/Z_m)) = d_{\mathfrak{P}_{\infty}}(T_m/Z_m) + d_{\mathfrak{Q}}(T_m/Z_m) = 2(q-2). $$
Hence we get $g_m^*=g_m/[T_m:Z_m] = g_m/(q-1)$.
\end{proof}

\begin{theorem}[Weierstra\ss\ semigroup of $\mathfrak{P}_{\infty}^*=\mathfrak{P}_{\infty}\cap Z_m$]
\ \label{5theorem3}

The place $\mathfrak{P}_{\infty}^*$ has the Weierstra\ss\ semigroup $\mathbb{H}_m^*=q\cdot \mathbb{H}_{m-1} \cup \{ n \geq c_m^*\}$ with conductor $c_m^*=c_m/(q-1)$, where $c_m$ denotes the conductor of the Weierstra\ss\ semigroup of $\mathfrak{P}_{\infty}$ in the norm-trace tower. 
\end{theorem}
\begin{proof} First of all the number of gaps $\tilde{g}_m^*$ in $\mathbb{H}_m^*$ coincides with the genus $g_m^*$. This can be checked as in \cite{PellStiTorWS} by substituting $g_m\mapsto g_m^*=g_m/(q-1)$ and $c_m\mapsto c_m^*=c_m/(q-1)$. Therefore $\mathbb{H}_m^*$ comes into consideration for the Weierstra\ss\ semigroup of $\mathfrak{P}_{\infty}^*$. For the rest of the proof it remains to show that $\mathbb{H}_m^*$ contains all pole numbers of $\mathfrak{P}_{\infty}^*$. We show that an integer $n$ is contained in $\mathbb{H}_m^*$ if and only if $n(q-1)$ is contained in $\mathbb{H}_m$. 

For $m=0$ there is nothing to show. Using induction we assume that the assertion is valid for $m-1$ with $m>0$. The integer $n$ is contained in $\mathbb{H}_m^*$ if and only if $n/q\in \mathbb{H}_{m-1}^*$ or $n\geq c_m^*=c_m/(q-1)$ holds. By the induction hypothesis this is equivalent to $n(q-1)/q \in \mathbb{H}_{m-1}$ or $n(q-1)\geq c_m$. Either way, $n(q-1)$ is contained in $\mathbb{H}_m$ and the induction is complete.

Now we can conclude that every pole number is contained in $\mathbb{H}_m^*$. For any pole number $t$ of $\mathfrak{P}_{\infty}^*$ we have $t(q-1)\in \mathbb{H}_m$, because $\mathfrak{P}_{\infty}^*$ totally ramifies in $T_m/Z_m$ with index $q-1$. With the above we finally get $t\in \mathbb{H}_m^*$. 
\end{proof}
\begin{remark} The following assertions hold for all intermediate towers $S_m/K$ with $T_m \geq S_m \geq Z_m$.
\begin{enumerate}
\item[\textup{(a)}] There is a divisor $r$ of $q-1$ with $S_m=K(x_0^r,\ldots,x_m^r)$ and $r=[T_m:S_m]$. For $s_i:=x_i^r$ the intermediate tower $S_m/K$ is generated by
$$ s_{i+1} \cdot \left( s_{i+1}^{(q-1)/r} + 1 \right)^r = \frac{s_i^q}{( s_i^{(q-1)/r} + 1 )^r} 
\qquad \textup{for } i=0,\ldots,m-1.$$
\item[\textup{(b)}] The tower $S_m/K$ of height $m$ has genus $g_m/r$.
\item[\textup{(c)}] The tower $S_m/K$ of height $m$ has 
$$ N_1(S_m/K) = (q^{m+1} + \varepsilon^*_m )\cdot (q-1)/r + k $$
rational places with $2q \leq k \leq q^2-q+2$.
\item[\textup{(d)}] An integer $n$ is a pole number of $\mathfrak{P}_{\infty}$ in $S_m/K$ if and only if $n\cdot r$ is a pole number of $\mathfrak{P}_{\infty}$ in $T_m/K$.
\end{enumerate}
\end{remark}

\section{The Full Automorphism Group}\label{FullAutomorphism}
In the sections \ref{DecompositonGroup} and \ref{Decomposition} we have computed a subgroup of $G_m=\textup{Aut}(T_m/K)$ which we conjecture to be the entire automorphism group of $T_m/K$. 

\begin{conjecture}\label{conjecture}
The automorphism group of the norm-trace tower with height $m\geq 1$ has order
$$ | \textup{Aut}(T_m/K) | = 
\left\{ \begin{array}{cl} 
2q^2(q-1) & \ \text{for}  \ q\geq 3 \ \text{odd} \ \ \text{ or } \  m=1 \\
q^3(q^2-1) & \ \text{for} \ q\geq 4 \ \text{even}  \ \text{and} \ m = 2 \\
2q^4(q-1) & \ \text{for} \ q\geq 4 \ \text{even} \ \text{and} \ m\geq 3.
\end{array}\right.
$$
\end{conjecture}
This extends Aleschnikov's result for $m=1$ in \cite{AleGB}. He proved that every rational place of $T_1/K$ outside $x_0^q+x_0$ has gap number $q$ and hence these places cannot be conjugated to $\mathfrak{P}_{\infty}$. For $m\geq 2$ and $q\geq 3$ it might also hold that $q^m$ is a gap number of the rational places outside $x_0^q+x_0$, but their Weierstra\ss\ groups are still unknown. With the computer algebra system \textsc{Magma} we checked that $q^m$ is a gap number outside $x_0^q+x_0$ for $m=2$ and $q=3,\ldots,9$. In  contrast we also checked that $q^m$ is a pole number outside $x_0^q+x_0$ for $q=2$ and $m=1,2,3$. Therefore $q=2$ is conjecturally an exceptional case. For $(m,q)=(1,2)$ the norm-trace tower is elliptic and hence its automorphism group is known and does not coincide with our exhibited group. For $m=2,3$ we queried \textsc{Magma} and received the result that $T_2/\mathbb{F}_4$ has $168$ automorphisms and $T_3/\mathbb{F}_4$ has $96$ automorphisms.
\\[0.1cm]
In this section we present a proof that the exhibited automorphisms generate the full automorphism group of the norm-trace tower with height $m=2$ in odd characteristic and $m=2,3,4$ in even characteristic respectively. We will use the Hurwitz formula for the relative genus \cite[Theorem III..4.12]{Sti} of $T_m/T_m^{G_m}$ in order to achieve bounds for the decomposition index $r_{\infty}$. 
First we consider the ramification of the field extensions
$$ T_m\geq T_{m-\varepsilon} \geq Z_{m-\varepsilon}=T_m^{G_m(\mathfrak{P}_{\infty})} \geq T_m^{G_m(\mathfrak{P}_{\infty}\prod\mathfrak{P}_a)} =Z_{m-\varepsilon-1}^1. $$

\begin{proposition}\label{6prop0} The extension $T_m/Z_{m-\varepsilon-1}^1$ is unramified outside $x_0^q+x_0$ for $m\geq 1 + \varepsilon$. 
\end{proposition}
\begin{proof} We already know that this assertion is true for the extension $T_m/Z_{m-\varepsilon}$ 
by Proposition \ref{2prop1} and the proof of Theorem \ref{5theorem1}. Hence it remains to consider the extension $Z_{m-\varepsilon}/Z_{m-\varepsilon-1}^1$. For abbreviation we substitute $m-\varepsilon \mapsto m$ and define $T_{m-1}^1:=K(x_1,\ldots,x_m)$. This field is isomorphic to $T_{m-1}$ as well as $Z_{m-1}$ is isomorphic to $Z_{m-1}^1$ under the map $x_{i-1}\mapsto x_i$ for $1\leq i\leq m$. By the proof of Theorem \ref{5theorem1} the ramifying places of $T_{m-1}/Z_{m-1}$ are the pole of $z_0$ and the common zero of $z_0,\ldots,z_{m-1}$. Using the isomorphism above we obtain that the pole of $z_1$ and the common zero of $z_1,\ldots,z_m$ are the ramifying places of $T_{m-1}^1/Z_{m-1}^1$. These places are contained in the support of $x_0^q+x_0$. Hence $Z_{m-1}/Z_{m-1}^1$ is unramified outside $x_0^q+x_0$, because $T_{m-1}/T_{m-1}^1$ is unramified outside $x_0^q+x_0$ by Proposition \ref{2prop1}. 
\end{proof}

By this proof we also get a key statement for the proof of Theorem \ref{6theorem1}.
\begin{corollary}\label{6prop2}\label{6cor1}
There is at least one wildly ramified place in $T_m/T_m^{G_m}$ not conjugated to $\mathfrak{P}_{\infty}$ for $m\geq 2\varepsilon$.
\end{corollary}
\begin{proof} A zero $\mathfrak{R}$ of $x_{m-\varepsilon}^{q-1}+1$ is totally ramified in $T_{m-\varepsilon}/T_{m-\varepsilon-1}^1$ and has a non-trivial relative degree in $T_m/T_{m-1}$ by Proposition \ref{2prop1}. Since $\mathfrak{R}$ is unramified in $T_{m-\varepsilon}/Z_{m-\varepsilon}$ as well as in $T_{m-\varepsilon-1}^1/Z_{m-\varepsilon-1}^1$ by the above proof, it is totally ramified in $Z_{m-\varepsilon}/Z_{m-\varepsilon-1}^1$. Hence the place $\mathfrak{R}$ is wildly ramified in $T_m/T_m^{G_m}$. Also it cannot be conjugated to $\mathfrak{P}_{\infty}$, because $\mathfrak{P}_{\infty}$ has relative degree $1$ in $T_m/T_m^{G_m}$. 
\end{proof}

Now we collect bounds for $r_{\infty}$ obtained by the results of section \ref{DecompositionTower}. 
\begin{proposition}\label{6prop1} The decomposition index $r_{\infty}$ of $\mathfrak{P}_{\infty}$ has one of the following properties:
\begin{enumerate}
\item[\textup{(a)}] $r_{\infty} = 2q$ or $q^2(q-1) \leq r_{\infty} \leq N_1(T_m)$ for odd $q$ and $m\geq 2$.
\item[\textup{(b)}] $r_{\infty} = q(q+1)$ or $q^2(q-1) \leq r_{\infty} \leq N_1(T_m)$ for even $q$ and $m=2$.
\item[\textup{(c)}] $r_{\infty} = 2q^2$ or $q^3(q-1) \leq r_{\infty} \leq N_1(T_m)$ for even $q$ or $m\geq 3$.
\end{enumerate}

\end{proposition}
\begin{proof} We assume that $\mathfrak{P}_{\infty}$ is conjugated to a rational place $\mathfrak{R}$ outside $x_0^q+x_0$.  Then $r_{\infty}$ is as large as the decomposition index of $\mathfrak{R}$ in $T_m/T_m^{G_m}$ and of course it is at most as large as the number of all rational places. 
\\[0.1cm] 
(a),(c) By Proposition \ref{6prop0} we obtain that $\mathfrak{R}$ is unramified and hence completely decomposed in $T_m/ Z_{m-\varepsilon-1}^1$. Therefore the decomposition index of $\mathfrak{R}$ is at least $[T_m:Z_{m-\varepsilon-1}^1]=q^{\varepsilon+1}(q-1)$.
\\[0.1cm]
(b) For even $q$ and $m=2$ we did not calculate the fixed field of $G_m(\mathfrak{P}_{\infty}\prod\mathfrak{P}_a)$. So we only use the fact that $\mathfrak{R}$ is completely decomposed in $T_2/Z_0$.
\end{proof}

\begin{theorem}\label{6theorem1} In nearly all cases Conjecture \ref{conjecture} is true for $m=1,2$. In even characteristic it is also true for $m=3,4$. The only exceptional case beside $q=2$ may occur for $(m,q)=(2,3)$.
\end{theorem}

\begin{proof} Let $\mathfrak{P}$ be a place of the fixed field $T_m^{G_m}$ ramifying in $T_m/T_m^{G_m}$. All its extensions have the same ramification index $e_{\mathfrak{P}}$ and different exponent $d_{\mathfrak{P}}$. The different degree of these extensions is 
$|G_m|\deg(\mathfrak{P})d_{\mathfrak{P}}/e_{\mathfrak{P}}$. So by the Hurwitz formula we get

\begin{align*} 2g_m-2 & = [T_m:T_m^{G_m}](2g_G-2) + \deg(\mathfrak{D}(T_m/T_m^{G_m})) \\
& = |G_m|(2g_G-2 + \sum\limits_{\mathfrak{P}} \delta_{\mathfrak{P}}\deg(\mathfrak{P})) \end{align*}

where $g_G$ denotes the genus of $T_m^{G_m}/K$ and $\delta_{\mathfrak{P}}=d_{\mathfrak{P}}/e_{\mathfrak{P}}$ denotes the ratio of the different exponent $d_{\mathfrak{P}}$ and ramification index $e_{\mathfrak{P}}$ of a place $\mathfrak{P}$ in $T_m^{G_m}$.  For any wildly ramifying place $\mathfrak{P}$ we have $\delta_{\mathfrak{P}}\geq 1$ and for any tamely ramifying place $\mathfrak{P}$ we have $1/2 \leq \delta_{\mathfrak{P}}< 1$. 
\\[0.1cm]
The place $\mathfrak{P}_{\infty}$ is wildly ramified with $\delta_{\mathfrak{P}_{\infty}} = \delta_{\infty}=(q^{\varepsilon +1}-2)/(q^{\varepsilon}(q-1))$. In the following we will collect upper bounds for $r_{\infty}$ in several cases.
\\[0.1cm]
\textit{Case 1:} It is $g_G\geq 1$. Then we get
$ 2g_m-2 \geq |G_m| \delta_{\infty} = r_{\infty} d_{\infty} $
and hence 
\begin{equation}\label{case1}
 r_{\infty} \leq (2g_m-2)/(q^{\varepsilon +1}-2).  
\end{equation}

\textit{Case 2:} It is $g_G=0$. Because of $g_m>0$ it follows that $\delta=\sum\nolimits_{\mathfrak{P}} \delta_{\mathfrak{P}}\deg(\mathfrak{P})) > 2$ and $\delta-\delta_{\infty}>0$ by the Hurwitz formula. 
Consequently any other place $\mathfrak{Q}$ not conjugated to $\mathfrak{P}_{\infty}$ is ramified in $T_m/T_m^{G_m}$. 
\\[0.1cm]
\textit{Case 2(a):} In this case $\mathfrak{Q}$ is either a wildly ramified place or a tamely ramified non-rational place. We also include the case that there are two tamely ramified rational places of $T_m^{G_m}$. Either way, $\delta-\delta_{\infty} \geq 1$ holds. This implies $|G_m|(\delta-2) \geq r_{\infty}(q^{\varepsilon}-2)$ and
\begin{equation}\label{case2a}
 r_{\infty} \leq (2g_m-2)/(q^{\varepsilon}-2).  
\end{equation}

\textit{Case 2(b):} In the last case $\mathfrak{Q}$ is the only tamely ramified rational place of $T_m^{G_m}$ and there are no ramified places other than $\mathfrak{P}_{\infty}$ and $\mathfrak{Q}$. So it holds $\delta=\delta_{\infty}+\delta_{\mathfrak{Q}} > 2 $ with $\delta_{\mathfrak{Q}}=(e-1)/e$. By this inequality we can estimate the ramification index $e$ of $\mathfrak{Q}$ and $e\geq q$ results for even $q$ and $e\geq q+2$ holds for odd $q$ resp. $e\geq 6$ for $q=3$. Actually $e=q$ is impossible because $\mathfrak{Q}$ is tamely ramified. Hence $e\geq q+1$ and $\delta\geq \delta_{\infty} + q/(q+1)$. For even $q$ we conclude 
$|G_m|(\delta-2) \geq |G_m|(2q^2-2q-2)/(q^2(q^2-1))$ and
\begin{equation}\label{case2b}
r_{\infty} \leq q(g_m-1)/(q^2-q-1).
\end{equation}
We omit the calculations for odd $q$ in this case since we will see that they are unnecessary.
\\[0.1cm]
Comparing the inequalities (\ref{case1}), (\ref{case2a}) and (\ref{case2b}) with the inequalities in Proposition \ref{6prop1} we can prove our hypothesis. For $m=2$ we have $2g_2-2=2q^3-2q^2-2q$ and $g_G=0$ (see Theorem \ref{5remark1}). So case $2$ holds. For odd $q$ case 2(a) actually holds by Corollary \ref{6prop2}. We get $r_{\infty}\leq 2q^2+2q+2+4/(q-2)$. This proves the hypothesis for odd $q\neq 3$ and $m=2$. For even $q$ and $m=2,3$ all cases imply our hypothesis. For even $q$ and $m=4$ case 2(b) is impossible by Corollary \ref{6prop2}, while the other cases leads to the verification of our hypothesis.
\end{proof}

\begin{remark} We queried \textsc{Magma} for the automorphisms of the decomposition tower. It seems that $Z_m/K$ has only two automorphisms namely those generated by the ''reflection'' automorphism $z_i\mapsto 1/z_{m-i}$.
 
\end{remark}

\textbf{Acknowledgement.} I am deeply grateful to Bernd Heinrich Matzat for his support of my studium and my diploma thesis. Many thanks to Florian He\ss\ for his helpful comments on this article. 

I acknowledge the support of the \textsc{Berlin Mathematical School}, which is funded by the German Science Foundation (DFG)
     as a Graduate School in the framework of the ``Excellence
     Initiative''.
\addcontentsline{toc}{chapter}{References}
\bibliographystyle {plain}
\bibliography{biblio}
\nocite{Lagemann,BezGarNGT}

\end{document}